\documentclass[11pt, reqno]{amsart}

\usepackage[english]{babel}
\usepackage[letterpaper,top=2cm,bottom=2cm,left=3cm,right=3cm,marginparwidth=1.75cm]{geometry}

\usepackage{amsmath}
\usepackage{graphicx}
\usepackage{amssymb}
\usepackage{algorithm2e}
\usepackage[colorlinks=true, allcolors=blue]{hyperref}
\usepackage{color}

\newtheorem{theorem}{Theorem}[section]
\newtheorem{lemma}[theorem]{Lemma}

\newtheorem{proposition}{Proposition}[section]

\theoremstyle{definition}

\theoremstyle{remark}

\numberwithin{equation}{section} \errorcontextlines=0

\title{On an optimal problem of bilinear forms}
\date{ }
\author{Naihuan Jing$^{1}$,  Yibo Liu$^{2}$,  Jiacheng Sun$^{2}$,  Chengrui Zhao$^{2}$,  Haoran Zhu$^{2}$}
\address{$^{1}$Department of Mathematics, North Carolina State University, Raleigh, NC 27695, USA}
\email{jing@ncsu.edu}
\address{$^{2}$College of Sciences, Northeastern University, Shenyang, Liaoning 110004, China}
\email{1431958945@qq.com, sjcsrc0927@163.com, 17761598896@163.com, whrzhu@outlook.com}

\begin{document}
	\maketitle
	
	\begin{abstract}
		We study an optimization problem originated from the Grothendieck constant. A generalized normal equation is proposed and analyzed. We establish a correspondence between solutions of the general normal equation and its dual equation. Explicit solutions are described for the two-dimensional case.
	\end{abstract}

	\section{Introduction}
	The Grothendieck constant $K_G$ \cite{G} is the smallest constant such that for every $d \in \mathbb{N}$ and every matrix $A=(a_{ij})$ (real or complex),
	$$\sup_{u_i,v_j \in B^{(d)}} \sum\limits_{ij}a_{ij}\left<u_i,v_j\right>\leqslant K_G\cdot\sup_{x_i,y_j=\pm 1}\sum\limits_{ij}a_{ij}x_iy_j,$$
	where $B^{(d)}=\{{\bf x}\in \mathbb R^d| \, ||{\bf x}||=1\}$ is the unit ball in $\mathbb{R}^d$. Despite much efforts, the value of the constant $K_G$ remains unknown \cite{P, H} for both the real and the complex cases. It is known that the Grothendieck constant is related to the Bell inequality in quantum nonlocality of
quantum computation \cite{T, A}.
The Grothendieck constant $K_G$ is also interpreted as the integrality gap of a natural semidefinite programming
relaxation for the so-called $K_{M,N}$-quadratic programming problem \cite{AN, KN}.
	
	In this paper, we study a simplified problem that can be cast as an optimization one using elementary methods.
The idea is to first formulate the (real) Grothendieck constant as some generalized eigenvalue problem and then
solve the simplified version.

\section{An optimal problem and generalized normal equations}
	
	We first recall some basic material on using the singular value decomposition (SVD) to solve maximum and minimum values of
quadratic form. %In this case the problem can also be solved using calculus.

	%\subsection{Case 1}
	Let $q(x)=\sum\limits_{i,j}a_{ij}x_ix_j=x^TAx$ be a real quadratic form, where $A=(a_{ij})$ is an $n\times n$ real symmetric matrix.
The maximum and the minimum of $q(x)$ on the sphere $|x|=1$ are given by the eigenvalues of $A$.

More generally, let $A$ be an $m\times n$ rectangular real matrix  and consider
the optimization problem for the blinear form
$B(x, y)=x^TAy$ subject to $|x|=|y|=1$.
\begin{lemma}\label{lem1} The maximal value of $B(x, y)=x^TAy$ subject to $|x|=|y|=c$
is $c^2\sigma_1(A)$, where $\sigma_1(A)$ is the largest singular value of $A$.
\end{lemma}
\begin{proof}
By the singular value decomposition, the matrix $A$ can be written as
$A=U^T\Lambda V$ for two orthogonal matrices $U$, $V$ and $\Lambda=diag(\sigma_1, \ldots, \sigma_r)$, where $\sigma_1 \ge \sigma_2 \ge \dots \ge \sigma_r>0$.
Then
for $X=Ux, Y=Vy$
$$x^TAy=X^T\Lambda Y=\sigma_1X_1Y_1+\sigma_2X_2Y_2+\dots+\sigma_rX_rY_r.$$
Consequently
\begin{align*}
\sigma_1X_1Y_1+\sigma_2X_2Y_2+\dots+\sigma_rX_rY_r &\le\sigma_1(|X_1||Y_1|+\cdots+|X_r||Y_r|)\\
&\le \sigma_1\sqrt{\sum_i|X_i|^2}\sqrt{\sum_i|Y_i|^2}= c^2\sigma_1
\end{align*}
Clearly $\pm c^2\sigma_1$ are retainable ($X_1=\pm c,Y_1=c, X_2=\dots=X_r=0$), so they are the maximum and minimum of $B(x, y)$ under the constraint $|x|=|y|=c$.
\end{proof}

%When $X_1=1,Y_1=1, X_2=0, y_2=0,\dots, X_r=0, y_r=0$. we can find $x$ and $y$ with the method of question 1 and $q(x)=\sigma_1$ is the maximum value we seek;

%When $X_1=-1,Y_1=1, X_2=0, y_2=0,\dots, X_r=0, y_r=0$, we can find $x$ and $y$ with the method of question 1 and $q(x)=\sigma_1$ is the minimum value we seek.

%\subsection*{Now let us consider a similar problem for the bilinear form $B(x,y)=x^TAy$.}

Now we consider the general situation. Let $B(\vec{x},\vec{y})$ be the bilinear function $(\mathbb R^d)^n\times (\mathbb R^d)^n\mapsto \mathbb R$:
%for $\vec{x_i},\vec{y_i} \in \mathbb C^d$
\begin{equation}\label{e:bilin}
B(\vec{x},\vec{y})=\sum_{k,j=1}^na_{kj}\vec{x}_k\cdot\vec{y}_j=\vec{x}^T\cdot A\vec{y}
\end{equation}
where $\vec{x}=(\vec{x}_1, \cdots, \vec{x}_n), \vec{y}=(\vec{y}_1, \cdots, \vec{y}_n)\in (\mathbb R^d)^n$ and we would like to compute its maximum under the constraint $|\vec{x}_k|=|\vec{y}_k|=1$. %so find the maximum and the minimum of $B(\vec{x},\vec{y})$.

%Let $B(\vec{x},\vec{y})=\sum_{i,j}a_{ij}\vec{x_i}\vec{y_j}$, where $|\vec{x_i}|=|\vec{y_i}|=1$. And
The variables
$\vec{x}_k=(x_{kr}), \vec{y}_k=(y_{kr})\in \mathbb R^d$ are points of the unit sphere: $\sum_{r=1}^dx_{kr}^2=\sum_{r=1}^dy_{kr}^2=1$. %$1\le r\le d$.
Let $F(x,y,\lambda,\mu)=B(\vec{x},\vec{y})-\dfrac{1}{2}\sum_{k=1}^n\lambda_k(|x_k|^2-1)-\dfrac{1}{2}\sum_{k=1}^n\mu_k(|y_k|^2-1)$, where
$\lambda_k, \mu_k$ are parameters.   Using the Lagrange multipliers, we have for any $1\le k\le n, 1\le r\le d$
\begin{align}\label{e:LM1}
%\dfrac{\partial F}{\partial x_{kr}}\dfrac{\partial F}{\partial y_{kr}}&
&\sum_{j=1}^na_{kj}y_{jr}-\lambda_kx_{kr}=0, \\ \label{e:LM2}
&\sum_{j=1}^na_{jk}x_{jr}-\mu_ky_{kr}=0.
\end{align}
Multiplying \eqref{e:LM1} by $x_{kr}$ and taking sums over $r$, we have
\begin{equation}
\sum_{r=1}^d(\sum_{j=1}^na_{kj}y_{jr})x_{kr}-\sum_{r=1}^d\lambda_kx^2_{kr}=0.
\end{equation}
Then $\lambda_k=\sum_{j=1}^na_{kj}\vec{x_k}^T\vec{y_j}$, subsequently at an optimal point $(\vec{x}, \vec{y})$
\begin{equation}
B(\vec{x}, \vec{y})=\sum_{k=1}^n\sum_{j=1}^na_{kj}\vec{x_k}^T\vec{y_j}=\sum_{k=1}^n\lambda_k.
\end{equation}
Similarly, $\mu_k=\sum_{j=1}^na_{jk}\vec{x_k}^T\vec{y_j}$ and we also have at the optimal point $(\vec{x}, \vec{y})$
\begin{equation}
B(\vec{x}, \vec{y})=\sum_{k=1}^n\sum_{j=1}^na_{kj}\vec{x_k}^T\vec{y_j}=\sum_{k=1}^n\mu_k,
\end{equation}
therefore the optimal value is $\sum_i\lambda_i=\sum_i\mu_i$.%$$\sum_{k=1}^n\mu_k=B(\vec{x},\vec{y})$$

To determine the extremal points $\vec{x}_k, \vec{y}_k$, assume that $\lambda_k\neq 0$ we can solve \eqref{e:LM1}
%Then we try to represent $x_{ki}, y_{kj}$ using
%$\dfrac{\partial F}{\partial x_{kr}}=\sum_{j=1}^na_{kj}y_{jr}-\lambda_kx_{kr}=0$
\begin{equation}\label{e:opt1}
x_{kr}=\dfrac{1}{\lambda_k}\sum_{j=1}^na_{kj}y_{jr}.
\end{equation}
It follows from plugging \eqref{e:opt1} into \eqref{e:LM2} that
%Plug this $x_{kr}$ into $\dfrac{\partial F}{\partial y_{kr}}$, then we can get
%$\sum_{i=1}^na_{ik}\left(\dfrac{1}{\lambda_i}\sum_{j=1}^na_{ij}y_{jr}\right)-\mu_ky_{kr}=0$
\begin{equation}
\sum_{i,j=1}^n\dfrac{1}{\lambda_i}a_{i k}a_{i j}y_{j r}-\mu_k y_{k r}=0,
\end{equation}
which can be written as
%$$\sum_{j=1}^n\left(\sum_{i=1}^n\frac{a_{i k}a_{i j}}{\lambda_i}\right)y_{j r}=\mu_k y_{k r}$$
\begin{equation}
\begin{pmatrix}a_{1 k}& a_{2k} & \dots&a_{nk}\end{pmatrix}\begin{pmatrix}{\lambda_1^{-1}}&0& \cdots& 0\\0&\lambda_2^{-1} & \cdots & 0
\\\vdots & \vdots& \ddots  & \vdots \\ 0 & 0 &\cdots &{\lambda_n^{-1}}
\end{pmatrix}A\begin{pmatrix}y_{1r}\\ y_{2r} \\ \vdots\\y_{nr}\end{pmatrix}=\mu_ky_{kr}.
\end{equation}
Combing all optimal points $\vec{y}_k=(y_{k1}, \cdots, y_{kd})$, $1\le k\le n$, we have that for $1\le r\le d$
\begin{equation}\label{e:gennorm1}
A^T\begin{pmatrix}{\lambda_1^{-1}}&0& \cdots& 0\\0&\lambda_2^{-1} & \cdots & 0
\\\vdots & \vdots& \ddots  & \vdots \\ 0 & 0 &\cdots &{\lambda_n^{-1}}
\end{pmatrix}A\begin{pmatrix}y_{1r}\\ y_{2r} \\\vdots\\y_{nr}\end{pmatrix}
=\begin{pmatrix}{\mu_1}&0& \cdots& 0\\0&\mu_2 & \cdots & 0
\\\vdots & \vdots& \ddots  & \vdots \\ 0 & 0 &\cdots &{\mu_n}
\end{pmatrix}\begin{pmatrix}y_{1r}\\y_{2r}\\\vdots\\y_{nr}\end{pmatrix}
\end{equation}

Similarly, we get the equation:
\begin{equation}\label{e:gennorm2}
A\begin{pmatrix}{\mu_1^{-1}}&0& \cdots& 0\\0&\mu_2^{-1} & \cdots & 0
\\\vdots & \vdots& \ddots  & \vdots \\ 0 & 0 &\cdots &{\mu_n^{-1}}
\end{pmatrix}A^T\begin{pmatrix}x_{1r}\\x_{12}\\\vdots\\x_{nr}\end{pmatrix}
=\begin{pmatrix}{\lambda_1}&0& \cdots& 0\\0&\lambda_2 & \cdots & 0
\\\vdots & \vdots& \ddots  & \vdots \\ 0 & 0 &\cdots &{\lambda_n}
\end{pmatrix} \begin{pmatrix} x_{1r} \\x_{12}\\ \vdots \\ x_{nr} \end{pmatrix}
\end{equation}

We call equations \eqref{e:gennorm1}-\eqref{e:gennorm2}
 the {\it generalized normal equations}, and the vectors $x_r=(x_{1r}, \ldots, x_{nr})$ and $y_r=(y_{1r}, \ldots, y_{dr})$ the {\it generalized
  singular vectors} associated with the pair of generalized (diagonal) singular matrices $\Lambda=diag(\lambda_1, \ldots, \lambda_n)$ and $M=diag(\mu_1, \ldots, \mu_n)$ such that $tr(\Lambda)=tr(M)$.

 \begin{proposition} The optimal point of the bilinear function $B(\vec{x}, \vec{y})=\sum_{k,j=1}^na_{kj}\vec{x}_k\cdot\vec{y}_j=\vec{x}^T\cdot A\vec{y}$ on
$(\mathbb R^d)^n$ under
the coonstraint $|\vec{x}_k|=|\vec{y}_k|=1 \ (1\le k\le n)$ obeys the following generalized normal equation
%\eqref{e:gennorm1}-\eqref{e:gennorm2}
%the generalized normal equations for the square matrix $A\in \mathbb R^{n\times n}$ are
\begin{align}\label{e:normal1}
	(A^T\Lambda^{-1}A-M)\begin{pmatrix}\vec{y}_1\\ \vdots \\ \vec{y}_n\end{pmatrix}&=0,\\ \label{e:normal2}
	(AM^{-1}A^T-\Lambda)\begin{pmatrix}\vec{x}_1\\ \vdots \\ \vec{x}_n\end{pmatrix}&=0.
\end{align}
and the optimal value is $tr(\Lambda)=tr(M)$.
\end{proposition}

The conditions \eqref{e:gennorm1} and \eqref{e:gennorm2} are equivalent, as shown as follows.

Suppose $(A^T\Lambda^{-1}A-M)x=0$ for some nonzero $x\in\mathbb R^n$, then $(A^T\Lambda^{-1})(AM^{-1})$ has eigenvalue $1$.
It follows that $(AM^{-1})(A^T\Lambda^{-1})$ also has eigenvalue $1$, so there exists nonzero $y\in\mathbb R^n$ such that
$(AM^{-1}A^T-\Lambda)y=0$. This shows that \eqref{e:gennorm1} and \eqref{e:gennorm2} are equivalent.

We remark that the diagonal entries $\Lambda$ and $M$ can be viewed as squares of generalized singular values. In fact,
if $\Lambda=M=\lambda I$, then \eqref{e:normal1} becomes
%$\lambda_1=\mu_1\geqslant 0,\lambda_2=\mu_2\geqslant0,\dots,\lambda_n=\mu_n\geqslant0$, denote $\gamma=\begin{pmatrix}\sqrt{\lambda_1}&&\\& \ddots &\\ & & \sqrt{\lambda_n} \end{pmatrix}$, and $\Lambda=M=\gamma\gamma^T=\gamma^T\gamma$, the generalized singular value equations can be changed as
\begin{align}
	(\lambda^{-1}A^TA-\lambda I)y=0,
%	(\mu^{-1}AA^T-\lambda I)x=0.
\end{align}
which implies that $\lambda$ is a singular value.

%We can get the following equation like get SVD according to the former equations:
%$\left(A\Lambda_1A^T-\Lambda_2\right)x=0,$ where $\Lambda_1, \Lambda_2$ are diagonal.
%
%If $x \neq 0$, then $\text{det}\left(A\Lambda_1A^T-\Lambda_2\right)=0$. It is the generalization of $\text{det}(A-\lambda E)=0$ eigenvalue equation.
%If $\text{det}(A) \neq 0$, we can get $\text{det}(A)\text{det}( \Lambda_1A^T-A^{-1}\Lambda_2)=0$, so $\text{det}(\Lambda_1A^T-A^{-1}\Lambda_2)=0$.
%More generally, we consider $\text{det}(\Lambda_1A-B\Lambda_2)=0$, which is called Generalized eigenvalue (or singular value) equation. One of the special case of this is $\text{det}(A-\lambda B)=0$.
%
%However, if we try to simplify the condition that $\vec{x_i}^T=\vec{y_i}^T$, We can perform the following derivation like this: $q(x)=\vec{x}^TA\vec{x}=\sum_{k,j}^na_{kj}\vec{x_k}\vec{x_j}$, and $|x|=1$, so we can use calculus theory like before.

\section{Two-dimensional cases}
%\subsection{Condition 1}
%\textbf{We consider the following special situations:}

The one-dimensional case problem concerns the optimal value of
$B(x,y)=x^TAy=\sum_{k,j=1}^{n}a_{kj}x_ky_j$ under the constraint $x_i,y_i \in \{1, -1\}$. This problem can be
solved by listing all possible values and pick up the maximum one.

Another simplified situation is the optimal problem of the blinear function
$B(\vec{x},\vec{y})=\vec{x}^TA\vec{y}=\sum_{k,j=1}^na_{kj}\vec{x_k}\vec{y_j}$, and $\vec{x_k},\vec{y_j} \in \mathbb{R}^d$ under the constraint $|\vec{x}|=|\vec{y}|=n$. As a result, we have the following estimate.
%so find the maximum and the minimum of $B(\vec{x},\vec{y})$.

\begin{theorem} Let $A$ be a real $n\times n$ matrix. The maximum and minimum of the bilinear function
$B(\vec{x},\vec{y})=\sum_{k,j=1}^na_{kj}\vec{x}_k\cdot\vec{y}_j$ under the constraint $|\vec{x}_k|=|\vec{y}_k|=1$ are bounded by
$\pm n^2d\sigma_1(A)$, where $\sigma_1(A)$ is the largest singular value of $A$.
\end{theorem}
\begin{proof} Note that the constraint $|\vec{x}_k|=|\vec{y}_k|=1$ is contained in the condition $|\vec{x}|=|\vec{y}|=n$,
where $\vec{x}$ and $\vec{y}$ are the juxtaposition of $\vec{x}_k$ and $\vec{y}_k$ respectively, i.e. $\vec{x}=(\vec{x}_1, \ldots, \vec{x}_n)$.
We order the coordinates of $\vec{x}_k$ (or $\vec{y}_k$) lexicographically follows: $x_{11}, x_{12}, \ldots, x_{1d}$, $x_{21}, \ldots, x_{2d}$,  $\ldots,
x_{n1}, \ldots, x_{nd}$ and similarly for $y_{jr}$. Let $\vec{u}$ (resp. $\vec{v}$) be the juxtaposition of $\vec{x}_1, \ldots, \vec{x}_n$
(resp. $\vec{y}_1, \ldots, \vec{y}_n$, then the bilinear function $B(\vec{x}, \vec{y})$ can be written as $B(\vec{u}, \vec{v})=\sum_{i, j=1}^{nd}B_{ij}u_iv_j$, where $B=A\otimes J$, where $J$ is the $d\times d$ matrix with all entries $1$. The constraint is now $|\vec{u}|=|\vec{v}|=n$.

By Lemma \ref{lem1} the maximal value of the bilinear function $B(\vec{u}, \vec{v})$ under the constraint $|\vec{u}|=|\vec{v}|=n$
is $n^2\sigma_1(B)$, where $\sigma_1(B)$ is the first singular value of $B$.
Note that the eigenvalues of $B=A\otimes J$ are the products of those of $A$ and $J$ counting with multiplicities. It is easy to see the
characteristic polynomial of $J$ is $f(t)=t^{d-1}(t-d)$. So the first eigenvalue of $BB^t=d(AA^t\otimes J)$ is $d^2\sigma_1(A)^2$. Therefore
the maximum value of $B(\vec{u}, \vec{v})$ is $n^2d\sigma(A)$, which implies that
\begin{equation}
\sup_{\vec{x}_i,\vec{y}_j\in B^{(d)}} \sum\limits_{i,j=1}^na_{ij}\left<\vec{x}_i,\vec{y}_j\right>\leqslant n^2d\sigma_1(A)
\end{equation}
\end{proof}

%We can find that these two questions are different from the two problems above:Frstly each component of a vector $\vec{x_k},\vec{y_j}$ is a vector; Secondly, the condition requires the length of each component of the vector $\vec{x_k},\vec{y_j}$ is 1 instead of the length of $\vec{x_k},\vec{y_j}$ is 1.
%
%However, we can completely solve these two problems,because we use the calculus theory and linear algebra method only to solve the extreme value problem under the condition that the length of $\vec{x_k},\vec{y_j}$ is n,which is necessary rather than sufficient condition for the condition that the length of each component of the vector $\vec{x_k},\vec{y_j}$ is 1. But we can use these methods to do result estimation for this extreme value problem.And we try to give three special cases,when matrix $A$ is $2\times2$ matrix and $\vec{x_k},\vec{y_j}$ are $2\times1$ matrix.
%\subsubsection{Special Case 1}

Let's consider the case of quadratic form $q(x)=\Vec{x}^TA\Vec{x}=\sum_{i,j=1}^{2}a_{ij}\vec{x_i}\cdot\vec{x_j}$ under the
constraint $|\vec{x_i}|=1$, where $\vec{x_i}$ are viewed as column vectors in $\mathbb{R}^{2}$ and $A=(a_{ij})_{2\times 2}$ is symmetric.

The equation to determine an optimal value is
\begin{equation}
\begin{pmatrix} a_{11} & a_{12}\\ a_{21} & a_{22} \end{pmatrix}\begin{pmatrix} x_{11} \\x_{21}\end{pmatrix}
=\begin{pmatrix} \lambda_1 & 0\\ 0 & \lambda_2 \end{pmatrix}\begin{pmatrix} x_{11} \\x_{21}\end{pmatrix}
\end{equation}
and the optimal value is given by $\lambda_1+\lambda_2$. To find the maximum value, we consider the
Lagrange multiplier $f(\lambda_1, \lambda_2)=\lambda_1+\lambda_2-\mu(\lambda_1\lambda_2-\lambda_1a_{22}-\lambda_2a_{11}+a_{11}a_{22}-a_{12}a_{21})$.
Then for $\mu\neq 0$ we have
\begin{align}\label{e:mu1}
1-\mu(\lambda_2-a_{22})&=0\Longrightarrow\lambda_2=\dfrac{1}{\mu}+a_{22}\\ \label{e:mu2}
	1-\mu(\lambda_1-a_{11})&=0\Longrightarrow \lambda_1=\dfrac{1}{\mu}+a_{11} \\ \label{e:mu3}
\lambda_1\lambda_2-\lambda_1a_{22}-\lambda_2a_{11}+&a_{11}a_{22}-a_{12}a_{21}=0
\end{align}
Plugging \eqref{e:mu1}-\eqref{e:mu2} into \eqref{e:mu3}, we have
$\dfrac{1}{\mu}=\sqrt{a_{12}a_{21}}$.
Therefore the maximum value is $f(\lambda_1, \lambda_2)=a_{11}+a_{22}+2\sqrt{a_{12}a_{21}}=a_{11}+a_{22}+2|a_{12}|$.
	
Now we consider the special case of diagonal bilinear form. Let
$B(\vec{x},\vec{y})=\sum_{i,j=1}^2a_{ij}\vec{x}_i^T\vec{y}_j$, where $A=diag(a_{11}, a_{22})$ and $\vec{x_i},\vec{y_i}$ are unit column vectors in $\mathbb{R}^{2}$.
	
It is easy to see that $\Lambda=M$ in this case, and the existence of generalized singular vectors implies that
\begin{align}\label{e:det1}
		&\text{det}\left(A\Lambda^{-1}A-\Lambda\right)=(\dfrac{a_{11}^2}{\lambda_1}-\lambda_1)(\dfrac{a_{22}^2}{\lambda_2}-\lambda_2),\\ \label{e:det2}
		&=\dfrac{a_{11}^2a_{22}^2}{\lambda_1\lambda_2}+\lambda_1\lambda_2-\dfrac{\lambda_2}{\lambda_1}a_{11}^2-\dfrac{\lambda_1}{\lambda_2}a_{22}^2.
	\end{align}
%Let $f(\lambda_{1},\lambda_{2})=\lambda_1+\lambda_2$ and
Consider the Lagrange multiplier:
\begin{equation} F(\lambda_1,\lambda_2,\mu)=\lambda_1+\lambda_2-\mu(\dfrac{a_{11}^2a_{22}^2}{\lambda_1\lambda_2}+\lambda_1\lambda_2-\dfrac{\lambda_2}{\lambda_1}a_{11}^2
-\dfrac{\lambda_1}{\lambda_2}a_{22}^2).
\end{equation}
Then
\begin{align}\label{e:LM4}
		1-\mu\left(-\dfrac{a_{11}^2a_{22}^2}{2\lambda_1^2\lambda_2}+\lambda_2+\dfrac{\lambda_2}{2\lambda_1^2}a_{11}^2
-\dfrac{1}{\lambda_2}a_{22}^2\right)=0\\ \label{e:LM5}
		1-\mu\left(-\dfrac{a_{11}^2a_{22}^2}{2\lambda_1\lambda_2^2}+\lambda_1-\dfrac{1}{\lambda_1}a_{11}^2+\dfrac{\lambda_2}{\lambda_1}^2a_{22}^2\right)=0\\ \label{e:LM6} \dfrac{a_{11}^2a_{22}^2}{\lambda_1\lambda_2}+\lambda_1\lambda_2-\dfrac{\lambda_2}{\lambda_1}a_{11}^2-\dfrac{\lambda_1}{\lambda_2}a_{22}^2=0
	\end{align}
	Plugging \eqref{e:LM6} into \eqref{e:LM4} and using \eqref{e:LM5} we get that %$$-\dfrac{a_{11}^2a_{22}^2}{2\lambda_1^2\lambda_2}+\lambda_2+\dfrac{\lambda_2}{2\lambda_1^2}a_{11}^2-\dfrac{1}{\lambda_2}a_{22}^2=-\dfrac{a_{11}^2a_{22}^2}{2\lambda_1\lambda_2^2}+\lambda_1-\dfrac{1}{\lambda_1}a_{11}^2+\dfrac{\lambda_2}{\lambda_1}^2a_{22}^2$$.
%	Then plug $\textcircled{3}$ into this equation to simplify, we can get a constraint that: $$\lambda_2-\dfrac{1}{\lambda_2}a_{22}^2=\lambda_1-\dfrac{1}{\lambda_1}a_{11}^2\dots\textcircled{4}.$$
%	Deform the $\textcircled{4}$, we can get
$$\begin{cases}
		\lambda_1\lambda_2-\dfrac{\lambda_1}{\lambda_2}a_{22}^2=\lambda_1^2-a_{11}^2
		\\a_{11}^2\dfrac{\lambda_2}{\lambda_1}-\dfrac{1}{\lambda_1\lambda_2}a_{11}^2a_{22}^2=a_{11}^2-\dfrac{1}{\lambda_2}a_{11}^4
	\end{cases}$$
	then plug them into \eqref{e:LM6}, and we get $$\dfrac{1}{\lambda_1^2}a_{11}^4+\lambda_1^2=2a_{11}^2,$$
	so $\lambda_1=\pm a_{11}$, similarly $\lambda_2=\pm a_{22}$. Therefore the maximum is given by
	$f(\lambda_1,\lambda_2)=\lambda_1+\lambda_2=|a_{11}\pm a_{22}|$.
	
	Now we consider the general symmetric case in two variables.
 Let $B(\vec{x},\vec{y})=\vec{x}^TA\vec{y}=\sum_{i,j=1}^2a_{ij}\vec{x_i}^T\vec{y_j}$, where $A=$ is symmetric,
$|\vec{x_i}|=|\vec{y_i}|=1$, and $\vec{x_i},\vec{y_j}$ are column vectors in  $\mathbb{R}^{2}$
	
	Since $A$ is symmetric, $\Lambda=M$. So the extremal value is $\lambda_1+\lambda_2$, and for $r=1, 2$
\begin{align}
A\begin{pmatrix}\frac{1}{\lambda_1}& 0\\0 &\frac{1}{\lambda_2}\end{pmatrix}A\begin{pmatrix}y_{1r}\\y_{2r}\end{pmatrix}
=\begin{pmatrix}\lambda_1&0 \\0 &\lambda_2\end{pmatrix}\begin{pmatrix}y_{1r}\\y_{2r}\end{pmatrix}\\
A\begin{pmatrix}\frac{1}{\lambda_1}&0\\0&\frac{1}{\lambda_2}\end{pmatrix}A\begin{pmatrix}x_{1r}\\x_{2r}\end{pmatrix}
=\begin{pmatrix}\lambda_1&0\\0&\lambda_2\end{pmatrix}\begin{pmatrix}x_{1r}\\x_{2r}\end{pmatrix}
\end{align}
	
Suppose $B(x, y)$ reaches an extremal value at $\vec{x}_i$, $\vec{y}_i$. This particular pair $\vec{x}_i$, $\vec{y}_i$ can be viewed as vectors. If the
span of the vectors is of the full rank, then $A\Lambda^{-1} A=\Lambda$.
%are Because of $A^T=A,\exists Q$ is orthogonal matrix subject to $A=Q^T\begin{pmatrix}s_1&\\&s_2\end{pmatrix}Q$, $$\left[Q^T\begin{pmatrix}s_1&\\&s_2\end{pmatrix}Q\begin{pmatrix}\frac{1}{\lambda_1}&&\\&&\frac{1}{\lambda_2}\end{pmatrix}Q^T\begin{pmatrix}s_1&\\&s_2\end{pmatrix}Q-\begin{pmatrix}\lambda_1&&\\&&\lambda_2\end{pmatrix}\right]\begin{pmatrix}x_{11}\\x_{2r}\end{pmatrix}=0$$
If $A$ is diagonal, then $a_{11}^2=\lambda_1^2, a_{22}^2=\lambda_2^2$, so $\lambda_i=\pm a_{ii}$ and
the maximum value is $|a_{11}\pm a_{22}|$, as we already obtained above.

Suppose $a_{12}\neq 0$, then the entries of $A$ satisfy the equations:
\begin{align}\label{e:A21}
&\lambda_1^{-1}a_{11}^2+\lambda_2^{-1}a_{12}^2=\lambda_1\\ \label{e:A22}
&\lambda_1^{-1}a_{12}^2+\lambda_2^{-1}a_{22}^2=\lambda_2\\ \label{e:A23}
&\lambda_1^{-1}a_{11}+\lambda_2^{-1}a_{22}=0.
\end{align}
Assuming $a_{11}a_{22}\neq 0$, then $\lambda_1/\lambda_2=-a_{11}/a_{22}$.
Plugging \eqref{e:A23} into \eqref{e:A21} leads to
$\lambda_1^2=\frac{a_{11}}{a_{22}}\det{A}$. Similarly $\lambda_2^2=\frac{a_{22}}{a_{11}}\det{A}$.
So the maximum is given by
$$\lambda_1+\lambda_2=\pm \sqrt{detA\frac{a_{11}}{a_{22}}}\frac{a_{11}-a_{22}}{a_{11}}$$
%We can eventually solve $\lambda_i$ in terms of $a_{ij}$to get the maximum.

Equation \eqref{e:A23} implies that if one of $a_{11}, a_{22}$ is zero, then the other must be zero. Suppose $a_{11}=a_{22}=0$, then $|B(\vec{x},\vec{y})|=|a_{12}\vec{x_1}^T\vec{y_2}+a_{21}\vec{x_2}^T\vec{y_1}|\leq |a_{12}|+|a_{21}|$. Therefore the
maximum value is $2|a_{12}|$, since it is attainable.
%largest possible value of $\lambda_1+\lambda_2\geq ?$
%{\color{red} Here $\lambda_1+\lambda_2$ is not bounded, so there are some problems}

Now we assume that the extremal point $(\vec{x}_i, \vec{y}_j)$ has proportional coordinates. In particular,
\begin{equation}
\det{\begin{pmatrix} x_{11} & x_{21}\\ x_{12}  & x_{22}\end{pmatrix}}=\det{\begin{pmatrix} y_{11} & y_{21}\\ y_{12}  & y_{22}\end{pmatrix}}=0
\end{equation}

Suppose $(x_{11}, x_{12})=k_1(x_{21}, x_{22})$ and $(y_{11}, y_{12})=k_2(y_{21}, y_{22})$, where the scaling factors $k_i$ must be the same otherwise
one can still derive that $A\Lambda^{-1} A=\Lambda$. In other words, we can assume that
the nullity of $A\Lambda^{-1} A-\Lambda$ is one, and $k_1=k_2=k=\pm 1$. So
\begin{align}
	&k(\lambda_1^{-1}a_{11}^2+\lambda_2^{-1}a_{12}^2-\lambda_1)+(\lambda_1^{-1}a_{11}a_{12}+\lambda_2^{-1}a_{12}a_{22})=0\\ \label{e:C11}
	&k(\lambda_1^{-1}a_{11}a_{12}+\lambda_2^{-1}a_{12}a_{22})+(\lambda_1^{-1}a_{12}^2+\lambda_2^{-1}a_{22}^2-\lambda_2)=0.
\end{align}
Then $\lambda_1=\pm(a_{12}+a_{11}k)/k$,$\lambda_2=\pm(a_{22}+a_{12}k)$. So the extremal value is
$$
\lambda_1+\lambda_2=\pm(a_{11}+a_{22}\pm 2a_{12}).
$$

\bibliographystyle{plain}

\end{document}